\documentclass[11pt,a4paper,oneside,reqno]{amsart}

\usepackage{latexsym}
\usepackage{amsmath}
\usepackage{amsthm}
\usepackage{amssymb}
\usepackage{amsfonts}
\usepackage{fancyhdr}
\usepackage{comment}

\usepackage{mathrsfs}                           % For different calligraphic fonts

\newcommand{\mR}{\mathbf{R}}                    % Formatting for R
\newcommand{\mC}{\mathbf{C}}                    % Formatting for C
\newcommand{\R}{\mathbf{R}}                    % Formatting for R
\newcommand{\C}{\mathbf{C}}                    % Formatting for C
\newcommand{\abs}[1]{\lvert #1 \rvert}          % Formatting for the absolute value
\newcommand{\norm}[1]{\lVert #1 \rVert}         % Formatting for the norm
\newcommand{\br}[1]{\langle #1 \rangle}         % Formatting for the inner product
\newcommand{\re}{\mathrm{Re}}
\newcommand{\im}{\mathrm{Im}}

\newcommand{\eps}{\varepsilon}

\newcommand{\closure}[1]{\overline{#1}}

\newcommand{\mOp}{\mathrm{Op}}

\theoremstyle{definition}
\newtheorem{thm}{Theorem}[section]
\newtheorem{prop}[thm]{Proposition}

\newtheorem{lemma}[thm]{Lemma}

\newtheorem*{remark}{Remark}

\title[Carleman estimates for Dirac]{Carleman Estimates and Inverse Problems for Dirac Operators}
\author{Mikko Salo and Leo Tzou}
\address{Department of Mathematics and Statistics \\ University of Helsinki}
\email{mikko.salo@helsinki.fi}
\address{Department of Mathematics \\ Stanford University}
\email{leo.tzou@gmail.com}

\thanks{M.S.~is supported by the Academy of Finland.}
\thanks{L.T.~is supported by the Doctoral Post-Graduate Scholarship from the Natural Science and Engineering Research Council of Canada. This article was written while L.T. was visiting the University of Helsinki and TKK, whose hospitality is gratefully acknowledged.}
\thanks{The authors would like to thank Andras Vasy and Lauri Ylinen for useful comments.}

\begin{document}

\begin{abstract}
We consider limiting Carleman weights for Dirac operators and prove corresponding Carleman estimates. In particular, we show that limiting Carleman weights for the Laplacian also serve as limiting weights for Dirac operators. As an application we consider the inverse problem of recovering a Lipschitz continuous magnetic field and electric potential from boundary measurements for the Pauli Dirac operator.
\end{abstract}

\maketitle

\section{Introduction}

In this article we consider Carleman estimates for Dirac operators, with applications to inverse problems. There is an extensive literature on Carleman estimates and their use in unique continuation problems (see for instance \cite{kochtataru}). However, Carleman estimates have also become an increasingly useful tool in inverse problems, we refer to \cite{isakov_carlemansurvey} for some developments.

The recent work \cite{ksu} shows the importance of Carleman estimates in the construction of complex geometrical optics (CGO) solutions to Schr\"odinger equations. CGO solutions have been central in establishing unique identifiability results for inverse problems for elliptic equations (starting with \cite{calderon} and \cite{sylvesteruhlmann}, see surveys \cite{uhlmannicm}, \cite{uhlmannselecta}). The approach of \cite{ksu} provides a general method of constructing CGO solutions, and has led to new results in inverse problems with partial data \cite{dksu}, \cite{knudsensalo}, determination of inclusions and obstacles \cite{uhlmannwang_elasticity}, \cite{heckuhlmannwang}, \cite{uhlmannwang_twodim_system}, \cite{salowang}, and also anisotropic inverse problems \cite{DKSaU}. All the results listed are concerned with, or can be reduced to, Schr\"odinger equations whose principal part is the Laplacian.

Our aim is to extend the Carleman estimate approach of \cite{ksu} to Dirac operators and their perturbations. By a Dirac operator we mean a self-adjoint, homogeneous, first order $N \times N$ matrix operator $P_0$ with constant coefficients in $\mR^n$, whose square is $-\Delta$. Examples include Dirac operators on differential forms in $\mR^n$, and the Pauli Dirac in $\mR^3$ given by the $4 \times 4$ matrix 
\begin{equation} \label{paulidirac_definition}
P_0(D) = \left( \begin{array}{cc} 0 & \sigma \cdot D \\ \sigma \cdot D & 0 \end{array} \right),
\end{equation}
where $D = -i\nabla$ and $\sigma = (\sigma_1, \sigma_2, \sigma_3)$ is a vector of Pauli matrices with 
\begin{equation*}
\sigma_1 = \left( \begin{array}{cc} 0 & 1 \\ 1 & 0 \end{array} \right), \ \sigma_2 = \left( \begin{array}{cc} 0 & -i \\ i & 0 \end{array} \right), \ \sigma_3 = \left( \begin{array}{cc} 1 & 0 \\ 0 & -1 \end{array} \right).
\end{equation*}

The first step is to prove Carleman estimates for Dirac with a special class of weights, called limiting Carleman weights. These weights were introduced in \cite{ksu} for the Laplacian (see Section \ref{sec:carleman} for the definition), and in this case include linear and logarithmic functions. See \cite{DKSaU} for a complete characterization. One may think of a limiting weight as a function $\varphi$ such that the Carleman estimate is valid both for $\varphi$ and $-\varphi$. A typical result for Dirac is as follows. Here and below $\norm{\,\cdot\,} = \norm{\,\cdot\,}_{L^2(\Omega)}$, and $A \lesssim B$ means that $A \leq CB$ where $C > 0$ is a constant independent of $h$ and $u$.

\begin{thm}
Let $\Omega \subseteq \mR^n$, $n \geq 2$, be a bounded open set, let $P_0$ be a Dirac operator as above, and let $V \in L^{\infty}(\Omega)^{N \times N}$. Let $\varphi$ be a smooth real-valued function with $\nabla \varphi \neq 0$ near $\closure{\Omega}$, and assume that $\varphi$ is a limiting Carleman weight for the Laplacian. If $h > 0$ is sufficiently small, then one has the estimate 
\begin{equation*}
h \norm{u} + h^2 \norm{\nabla u} \lesssim \norm{e^{\varphi/h}(P_0(hD)+hV)e^{-\varphi/h} u},
\end{equation*}
for any $u \in C^{\infty}_c(\Omega)^N$.
\end{thm}

We give two different proofs of such estimates. The first proof uses the fact that Dirac squared is the Laplacian, and gives the estimate for any limiting Carleman weight for the Laplacian. This proof, however, is based on symbol calculus and does not give Carleman estimates with boundary terms, which were crucial in the partial data result of \cite{ksu}. We will present another proof, which uses just integration by parts and gives a simple Carleman estimate with boundary terms. The second proof is valid for the linear weight.

The next step is to construct CGO solutions to Dirac equations by using the Carleman estimate. These solutions are then used to prove unique identifiability of the coefficients from boundary measurements. We carry out this program for an inverse problem for the Pauli Dirac operator with magnetic and electric potentials, which we set out to define.

Let $P_0$ be as in \eqref{paulidirac_definition} and consider the operator 
\begin{equation} \label{lv_definition}
\mathcal{L}_V = P_0(D) + V,
\end{equation}
where the potential $V$ has the form 
\begin{equation} \label{v_definition}
V = P_0(A) + Q,
\end{equation}
with $Q = \left( \begin{smallmatrix} q_+ I_2 & 0 \\ 0 & q_- I_2 \end{smallmatrix} \right)$, $A = (a_1,a_2,a_3) \in W^{1,\infty}(\Omega ; \mR^3)$, and $q_{\pm} \in W^{1,\infty}(\Omega)$. We consider $\mathcal{L}_V$ acting on $4$-vectors $u = \left( \begin{smallmatrix} u_+ \\ u_- \end{smallmatrix} \right)$ where $u_{\pm} \in L^2(\Omega)^2$. It is shown in \cite{nakamuratsuchida} that the boundary value problem 
\begin{equation*}
\left\{ \begin{array}{rll}
\mathcal{L}_V u &\!\!\!= 0 & \quad \text{in } \Omega, \\
u_+ &\!\!\!= f & \quad \text{on } \partial \Omega,
\end{array} \right.
\end{equation*}
is well posed if $\Omega$ has smooth boundary and $0$ is not in the spectrum of $\mathcal{L}_V$. There is a unique solution $u \in \mathcal{H}(\Omega)^2$ for any $f \in h(\partial \Omega)$, where $\mathcal{H}(\Omega) = \{u \in L^2(\Omega)^2 \,;\, \sigma \cdot Du \in L^2(\Omega)^2 \}$ and $h(\partial \Omega) = \mathcal{H}(\Omega)|_{\partial \Omega}$. We refer to \cite{nakamuratsuchida} for more details and a description of the trace space $h(\partial \Omega)$.

There is a well-defined Dirichlet-to-Dirichlet map 
\begin{equation} \label{ddmap}
\Lambda_V: h(\partial \Omega) \to h(\partial \Omega), \ f \mapsto u_-|_{\partial \Omega}.
\end{equation}
Thus, $\Lambda_V$ takes $u_+$ to $u_-$ on the boundary, where $u$ is the solution of the Dirichlet problem above. This map is invariant under gauge transformations: if $A$ is replaced by $A + \nabla p$ where $p|_{\partial \Omega} = 0$, the map $\Lambda_V$ stays unchanged and so does the magnetic field $\nabla \times A$.

More generally, if $\Omega$ is any bounded open set in $\mR^n$, we define the Cauchy data set 
\begin{equation*}
C_V = \{ (u_+|_{\partial \Omega}, u_-|_{\partial \Omega}) \,;\, u \in \mathcal{H}(\Omega)^2 \text{ and } \mathcal{L}_V u = 0 \text{ in } \Omega \}.
\end{equation*}
Here the traces are taken in the abstract sense as elements of $\mathcal{H}(\Omega)/H^1_0(\Omega)$. If $\partial \Omega$ is smooth and $0$ is outside the spectrum, $C_V$ is just the graph of $\Lambda_V$. Also here $C_V$ is unchanged in the gauge transformation $A \mapsto A + \nabla p$ if $p|_{\partial \Omega} = 0$.

We consider $C_V$ as the boundary measurements, and the inverse problem is to determine the magnetic field $\nabla \times A$ and the electric potentials $q_{\pm}$ from $C_V$. This inverse problem, and the correspoding inverse scattering problem at fixed energy, have been studied in \cite{isozakidirac}, \cite{gotodirac}, \cite{nakamuratsuchida}, \cite{lidirac} with varying assumptions on the potential. The next uniqueness result was proved for smooth coefficients in \cite{nakamuratsuchida} by reducing to a second order equation and using pseudodifferential conjugation. We establish this result for Lipschitz continuous coefficients by working with the Dirac system directly using the Carleman estimate approach. In particular, the proof avoids the pseudodifferential conjugation argument.

\begin{thm} \label{thm:uniqueness}
Let $\Omega \subseteq \mR^3$ be a bounded $C^1$ domain, let $A_j \in W^{1,\infty}(\Omega ; \mR^3)$, and let $q_{\pm,j} \in W^{1,\infty}(\Omega)$, $j = 1,2$. If $C_{V_1} = C_{V_2}$, then $\nabla \times A_1 = \nabla \times A_2$ and $q_{\pm,1} = q_{\pm,2}$ in $\Omega$.
\end{thm}

The first step in the proof is a boundary determination result, which again was proved in \cite{nakamuratsuchida} for smooth coefficients and domains by pseudodifferential methods but which is needed here in the nonsmooth case. We use arguments of \cite{brownboundary}, \cite{brownsalo} given for scalar equations, and construct solutions to the Dirac equation which concentrate near a boundary point and oscillate at the boundary. These solutions can be used to find the values of the tangential component $A_{\text{tan}} = A - (A\cdot \nu)\nu$ of $A$ and of $q_{\pm}$ at the boundary. Here $\nu$ is the outer unit normal to $\partial \Omega$.

\begin{thm} \label{thm:boundary}
Let $\Omega \subseteq \mR^3$ be a bounded $C^1$ domain and $A, q_{\pm} \in W^{1,\infty}(\Omega)$. Then $C_V$ uniquely determines $A_{\text{tan}}|_{\partial \Omega}$ and $q_{\pm}|_{\partial \Omega}$.
\end{thm}

Besides their independent interest, a major motivation for studying Dirac operators comes from inverse problems for other elliptic systems. These problems have often been treated by reducing the system to a second order Schr\"odinger equation, such as with the Maxwell equations \cite{olasomersalo} or linear elasticity \cite{nakamurauhlmann}, \cite{nakamurauhlmannerratum}, \cite{eskinralston_elasticity}. These reductions are however not without problems, and for instance counterparts of the partial data result in \cite{ksu} for systems may be difficult to prove in this way. For the Maxwell equations, there is another useful reduction to a $8 \times 8$ Dirac system (see \cite{olasomersalo}) which may be more amenable to a partial data result.

Another benefit of working with first order systems directly is that the reduction to a second order system may require extra derivatives of the coefficients. For instance, the reduction of Maxwell equations to a Schr\"odinger equation in \cite{olasomersalo} requires two derivatives of the coefficients, while the Dirac system only requires one derivative. Also, the Carleman estimates are valid for low regularity coefficients. Theorem \ref{thm:uniqueness}, whose proof uses the first order structure and Carleman estimates, may be thought of as an example result with improved regularity assumptions.

For other work on Carleman estimates and unique continuation for Dirac operators, we refer to \cite{berthier_dirac}, \cite{jerison_dirac}, \cite{mandache_dirac}. Also, while writing this article, we became aware of the work of M.~Eller \cite{eller_aip}, which also considers Carleman estimates for first order systems via integration by parts arguments.

The structure of the article is as follows. In Section 2 we give different proofs of Carleman estimates for Dirac with limiting weights. Section 3 gives a construction of CGO solutions for \eqref{lv_definition} with smooth coefficients, and the case of Lipschitz coefficients is considered in Section 4. The last two sections prove the uniqueness results for the inverse problem, Theorems \ref{thm:uniqueness} and \ref{thm:boundary}.

\section{Carleman estimates} \label{sec:carleman}

Let $P_0(\xi)$ be an $N \times N$ matrix depending on $\xi \in \mR^n$, such that each element of $P_0(\xi)$ is of the form $a \cdot \xi$ where $a \in \mC^n$. We assume the two conditions 
\begin{equation*}
P_0(\xi)^2 = \xi^2 I_N, \quad P_0(\xi)^* = P_0(\xi),
\end{equation*}
for all $\xi \in \mR^n$. These conditions imply that the operator $P_0(hD)$, where $h > 0$ is a small parameter, is a self-adjoint semiclassical Dirac operator:
\begin{equation*}
P_0(hD)^2 = (-h^2 \Delta) I_N, \quad P_0(hD)^* = P_0(hD).
\end{equation*}
We may as well define $P_0(\xi)$ for $\xi \in \mC^n$. Then $P_0(\xi)^* = P_0(\bar{\xi})$, but one still has $P_0(\xi)^2 = \xi^2 I_N$. This and the linearity of $P_0$ imply 
\begin{equation} \label{pzero_zeta_xi}
P_0(\zeta)P_0(\xi) + P_0(\xi)P_0(\zeta) = 2(\zeta \cdot \xi)I_N \quad \text{for } \zeta, \xi \in \mC^n.
\end{equation}

In this section $\Omega$ will be a bounded open set in $\mR^n$, where $n \geq 2$. We wish to prove a Carleman estimate, which is a bound from below for the conjugated operator 
\begin{equation*}
P_{0,\varphi} = e^{\varphi/h} P_0(hD) e^{-\varphi/h}.
\end{equation*}
It is well known that such bounds exist if the weight $\varphi$ enjoys some pseudoconvexity properties. On the other hand, in the applications to inverse problems one needs this estimate both for $\varphi$ and $-\varphi$, and the weight can only satisfy a degenerate pseudoconvexity assumption. This is the case for the limiting Carleman weights introduced in \cite{ksu}, where also a convexification procedure was given for obtaining Carleman estimates for these special weights.

We recall that $\varphi \in C^{\infty}(\tilde{\Omega} \,;\, \mR)$ is a limiting Carleman weight for the Laplacian in the open set $\tilde{\Omega}$, where $\Omega \subset \subset \tilde{\Omega}$, if $\nabla \varphi \neq 0$ and $\{a,b\} = 0$ when $a = b = 0$, where $a$ and $b$ are the real and imaginary parts, respectively, of the semiclassical Weyl symbol of the conjugated scalar operator $e^{\varphi/h} (-h^2 \Delta) e^{-\varphi/h}$. Examples are the linear weight $\varphi(x) = \alpha \cdot x$ where $\alpha \in \mR^n$, $\abs{\alpha} = 1$, and the logarithmic weight $\varphi(x) = \log\,\abs{x-x_0}$ where $x_0 \notin \Omega$.

One can deduce a Carleman estimate for the perturbed Dirac operator $P_0(hD) + hV$ directly from a corresponding estimate for the Laplacian. We use the semiclassical Sobolev spaces  
\begin{equation*}
\norm{u}_{H^s_{\text{scl}}} = \norm{\br{hD}^s u}_{L^2}
\end{equation*}
where $\br{\xi} = (1+\abs{\xi}^2)^{1/2}$.

The estimate for Laplacian is as follows. This was proved in \cite{ksu} with a gain of one derivative, i.e.~$\norm{u}_{H^{s+1}_{\text{scl}}}$ controlled by the right hand side. We will give a slightly different proof using the G{\aa}rding inequality to obtain a gain of two derivatives, which will give an improved estimate for Dirac. Below, we will use the conventions of semiclassical calculus. In particular we consider the usual semiclassical symbol classes $S^m$, and relate to symbols $a \in S^m$ operators $A = \mOp_h(a)$ via Weyl quantization. See \cite{dimassisjostrand} for more details. Also, we write $(u|v) = \int_{\Omega} u \cdot \bar{v} \,dx$ and $\norm{u} = (u|u)^{1/2}$.

\begin{lemma}
Let $\varphi$ be a limiting Carleman weight for the Laplacian, and let $\varphi_{\varepsilon} = \varphi + \frac{h}{\varepsilon} \frac{\varphi^2}{2}$ be a convexified weight. Then for $h \ll \varepsilon \ll 1$ and $s \in \mR$, 
\begin{equation} \label{laplacian_estimate}
\frac{h}{\sqrt{\varepsilon}} \norm{u}_{H^{s+2}_{\text{scl}}} \leq C \norm{e^{\varphi_{\varepsilon}/h} (-h^2 \Delta) I_N e^{-\varphi_{\varepsilon}/h} u}_{H^s_{\mathrm{scl}}}
\end{equation}
for all $u \in C_c^{\infty}(\Omega)^N$.
\end{lemma}
\begin{proof}
First extend $\varphi$ smoothly outside $\Omega$ so that $\varphi \equiv 1$ outside a large ball. In this way one can use Weyl symbol calculus in $\mR^n$, and the choice of extension does not affect the final estimates since the differential operators are only applied to functions supported in $\Omega$.

Let $P_{\varphi} = e^{\varphi/h} (-h^2 \Delta) e^{-\varphi/h}$, and write $P_{\varphi} = A + iB$ with $A$ and $B$ self-adjoint, so $A = -h^2 \Delta - \abs{\nabla \varphi}^2$, $B = \nabla \varphi \circ hD + hD \circ \nabla \varphi$. If $v \in C^{\infty}_c(\Omega)$ we integrate by parts to get 
\begin{equation} \label{v_integrate_parts}
\norm{P_{\varphi} v}^2 = ((A+iB)v|(A+iB)v) = \norm{Av}^2 + \norm{Bv}^2 + (i[A,B]v|v).
\end{equation}
Let the semiclassical Weyl symbols of $A$ and $B$ be given by $a = \xi^2 - \abs{\nabla \varphi}^2$ and $b = 2\nabla \varphi \cdot \xi$. We recall the composition formula for  semiclassical symbols $p$ and $q$: the operator $R = PQ$ has symbol 
\begin{equation*}
r \sim \sum_{\alpha,\beta} \frac{h^{\abs{\alpha+\beta}}(-1)^{\abs{\alpha}}}{(2i)^{\abs{\alpha+\beta}}\alpha!\beta!} (\partial_x^{\alpha} \partial_{\xi}^{\beta} p(x,\xi)) (\partial_{\xi}^{\alpha} \partial_x^{\beta} q(x,\xi)) = pq + \frac{h}{2i} \{p,q\} + O(h^2).
\end{equation*}
The commutator is given by 
\begin{equation*}
i[A,B] = \mOp_h(h\{a,b\}).
\end{equation*}

We now replace $\varphi$ by $\varphi_{\varepsilon} = f(\varphi)$ where $f(\lambda) = \lambda + \frac{h}{\varepsilon} \frac{\lambda^2}{2}$. Let $A_{\varepsilon}$ and $B_{\varepsilon}$ be the corresponding operators with symbols $a_{\varepsilon}$ and $b_{\varepsilon}$, and let $\eta = f'(\varphi)\xi$. A computation in \cite{ksu} implies that whenever $a_{\varepsilon}(x,\eta) = b_{\varepsilon}(x,\eta) = 0$ one has 
\begin{equation*}
\{a_{\varepsilon},b_{\varepsilon}\}(x,\eta) = 4 f''(\varphi) f'(\varphi)^2 \abs{\nabla \varphi}^4 + f'(\varphi)^3 \{a,b\}(x,\xi).
\end{equation*}
For the following facts on symbols we assume that $x$ is near $\closure{\Omega}$. The limiting Carleman condition and some algebra (see \cite{ksu}, \cite{knudsensalo}) show that 
\begin{equation*}
\{a_{\varepsilon},b_{\varepsilon}\}(x,\eta) = 4 f''(\varphi) f'(\varphi)^2 \abs{\nabla \varphi}^4 + m_{\varepsilon}(x) a_{\varepsilon}(x,\eta) + l_{\varepsilon}(x,\eta) b_{\varepsilon}(x,\eta),
\end{equation*}
where 
\begin{equation*}
m_{\varepsilon}(x) = -4 f'(\varphi) \frac{\varphi'' \nabla \varphi \cdot \nabla \varphi}{\abs{\nabla \varphi}^2}, \quad l_{\varepsilon}(x,\eta) = \left( \frac{4 \varphi'' \nabla \varphi}{\abs{\nabla \varphi}^2} + \frac{2 f''(\varphi) \nabla \varphi}{f'(\varphi)} \right) \cdot \eta.
\end{equation*}
The symbol of $i[A_{\eps},B_{\eps}]$ is then given by 
\begin{equation*}
h\{a_{\eps},b_{\eps}\} = \frac{4h^2}{\varepsilon} f'(\varphi)^2 \abs{\nabla \varphi}^4 + h m_{\varepsilon} a_{\varepsilon} + h l_{\varepsilon} b_{\varepsilon}.
\end{equation*}
The first term is always positive near $\closure{\Omega}$ since $\nabla \varphi$ is nonvanishing. To obtain a $H^2_{\text{scl}}$ bound we introduce the symbol 
\begin{equation*}
d_{\varepsilon} = 4 f'(\varphi)^2 \abs{\nabla \varphi}^4 + a_{\eps}^2 + b_{\eps}^2,
\end{equation*}
and we write the earlier identity as 
\begin{equation} \label{commsymbol_new}
h\{a_{\eps},b_{\eps}\} = \frac{h^2}{\varepsilon} d_{\eps} + h m_{\varepsilon} a_{\varepsilon} + h l_{\varepsilon} b_{\varepsilon} - \frac{h^2}{\varepsilon} (a_{\eps}^2 + b_{\eps}^2).
\end{equation}

Suppose that $h \ll \eps \ll 1$. Since $a_{\varepsilon}$ is elliptic and of order $2$, there is a constant $c_0 > 0$, independent of $\eps$, such that 
\begin{equation*}
d_{\varepsilon}(x,\eta) \geq c_0 \br{\eta}^4, \quad x \text{ near } \closure{\Omega}, \ \eta \in \mR^n.
\end{equation*}
The easy G{\aa}rding inequality implies that for $h$ small, 
\begin{equation*}
(D_{\varepsilon} v|v) \geq \frac{c_0}{2} \norm{v}_{H^2_{\text{scl}}}^2, \quad v \in C_c^{\infty}(\Omega).
\end{equation*}
On the other hand, we have the quantizations 
\begin{align*}
\mOp_h(a_{\eps}^2) &= A_{\eps}^2 + h^2 q_1(x), \\
\mOp_h(b_{\eps}^2) &= B_{\eps}^2 + h^2 q_2(x), \\
\mOp_h(m_{\eps} a_{\eps}) &= \frac{1}{2} m_{\eps} \circ A_{\eps} + \frac{1}{2} A_{\eps} \circ m_{\eps} + h^2 q_3(x), \\
\mOp_h(l_{\eps} b_{\eps}) &= \frac{1}{2} L_{\eps} B_{\eps} + \frac{1}{2} B_{\eps} L_{\eps} + h^2 q_4(x),
\end{align*}
where the $q_j$ are smooth functions which, together with their derivatives, are bounded uniformly with respect to $\eps$ near $\closure{\Omega}$. It follows from \eqref{commsymbol_new} that 
\begin{multline*}
(i[A_{\eps},B_{\eps}]v|v) = \frac{h^2}{\eps} (D_{\eps}v|v) + \frac{h}{2} \big[(A_{\eps}v|m_{\eps}v) + (m_{\eps}v|A_{\eps}v) \\
 + (B_{\eps}v|L_{\eps}v) + (L_{\eps}v|B_{\eps}v)\big] - \frac{h^2}{\eps} (\norm{A_{\eps}v}^2 + \norm{B_{\eps}v}^2) - h^3 (q_5 v|v) \\
 \geq \frac{c_0}{2} \frac{h^2}{\eps} \norm{v}_{H^2_{\text{scl}}}^2 - C h \norm{A_{\eps}v} \,\norm{v} - C h \norm{B_{\eps} v} \,\norm{v}_{H^1_{\text{scl}}} \\
  - \frac{h^2}{\eps} (\norm{A_{\eps}v}^2 + \norm{B_{\eps}v}^2) - C h^3 \norm{v}^2.
\end{multline*}
Recall that $h \ll \eps \ll 1$, where $\eps$ is fixed (depending on the constants $C$). Using the inequality $\abs{\alpha \beta} \leq \delta \abs{\alpha}^2 + \frac{1}{4 \delta} \abs{\beta}^2$ we obtain 
\begin{equation*}
(i[A_{\eps},B_{\eps}]v|v) \geq \frac{c_0}{4} \frac{h^2}{\eps} \norm{v}_{H^2_{\text{scl}}}^2 - \frac{1}{2} \norm{A_{\eps}v}^2 - \frac{1}{2} \norm{B_{\eps}v}^2.
\end{equation*}
Now \eqref{v_integrate_parts} shows that 
\begin{equation} \label{laplacian_unshifted}
\norm{P_{\varphi_{\varepsilon}} v}^2 \geq \frac{c_0 h^2}{4 \varepsilon} \norm{v}_{H^2_{\text{scl}}}^2.
\end{equation}

Finally, we need to shift \eqref{laplacian_unshifted} to a different Sobolev index to prove \eqref{laplacian_estimate}. Let $\Omega \subset \subset \tilde{\Omega}$ with $\varphi$ a limiting Carleman weight in $\tilde{\Omega}$, and let $\chi \in C_c^{\infty}(\tilde{\Omega})$ with $\chi = 1$ near $\closure{\Omega}$. If $u \in C_c^{\infty}(\Omega)$, we have the pseudolocal estimate 
\begin{equation*}
\norm{(1-\chi) \br{hD}^t u}_{H^{\alpha}_{\text{scl}}} \leq C_M h^M \norm{u}_{H^{\beta}_{\text{scl}}}, \quad \text{for any } \alpha,\beta,t,M.
\end{equation*}
This and \eqref{laplacian_unshifted} imply 
\begin{align*}
h \norm{\br{hD}^{s+2}u} &= h \norm{\br{hD}^s u}_{H^2_{\text{scl}}} \\
 &\leq h \norm{\chi \br{hD}^s u}_{H^2_{\text{scl}}} + h \norm{(1-\chi) \br{hD}^s u}_{H^2_{\text{scl}}} \\
 &\lesssim \sqrt{\varepsilon} \norm{P_{\varphi_{\varepsilon}} (\chi \br{hD}^s u)} + h^2 \norm{\br{hD}^{s+2} u}.
\end{align*}
By the pseudolocal property we have 
\begin{equation*}
\norm{[P_{\varphi_{\varepsilon}}, \chi] \br{hD}^s u} \leq C_M h^M \norm{\br{hD}^{s+2} u}, \quad \text{for any } M,
\end{equation*}
and since $\chi [P_{\varphi_{\varepsilon}}, \br{hD}^s] = h \mOp_h(S^s)$, we have for $h \ll \varepsilon \ll 1$ that 
\begin{equation*}
\norm{\chi [P_{\varphi_{\varepsilon}}, \br{hD}^s] u} \lesssim h \norm{\br{hD}^s u}.
\end{equation*}
One can check that all constants here can be chosen independent of $\varepsilon$. From these estimates we obtain \eqref{laplacian_estimate}.
\end{proof}

Our first estimate for Dirac follows immediately.

\begin{lemma} \label{carleman_estimate}
Let $V \in L^{\infty}(\Omega)^{N \times N}$ and let $\varphi$ be a limiting Carleman weight for the Laplacian. If $-1 \leq s \leq 0$, then for $h$ small one has 
\begin{equation*}
h \norm{u}_{H^{s+1}_{\text{scl}}} \lesssim \norm{e^{\varphi/h}(P_0(hD) + hV) e^{-\varphi/h} u}_{H^s_{\text{scl}}}
\end{equation*}
for all $u \in C_c^{\infty}(\Omega)^N$.
\end{lemma}
\begin{proof}
Consider first $P_{0,\varphi_{\varepsilon}} = e^{\varphi_{\varepsilon}/h} P_0(hD) e^{-\varphi_{\varepsilon}/h}$. Now \eqref{laplacian_estimate} implies 
\begin{equation} \label{carleman_freedirac_convexified}
\frac{h}{\sqrt{\varepsilon}} \norm{u}_{H^{s+1}_{\text{scl}}} \leq C \norm{P_{0,\varphi_{\varepsilon}}^2 u}_{H^{s-1}_{\mathrm{scl}}} = C \norm{\br{hD}^{-1} P_{0,\varphi_{\varepsilon}} P_{0,\varphi_{\varepsilon}} u}_{H^s_{\text{scl}}} \leq C \norm{P_{0,\varphi_{\varepsilon}} u}_{H^s_{\text{scl}}}
\end{equation}
since $\br{hD}^{-1} P_{0,\varphi_{\varepsilon}}$ is of order $0$. We add the perturbation $h V$ and use the estimate $\norm{Vu}_{H^s_{\text{scl}}} \leq \norm{Vu} \leq \norm{V}_{L^{\infty}} \norm{u} \leq \norm{V}_{L^{\infty}} \norm{u}_{H^{s+1}_{\text{scl}}}$ to get 
\begin{equation*}
\frac{h}{\sqrt{\varepsilon}} \norm{u}_{H^{s+1}_{\text{scl}}} \leq C(\norm{(P_{0,\varphi_{\varepsilon}} + hV)u}_{H^s_{\text{scl}}} + h \norm{V}_{L^{\infty}} \norm{u}_{H^{s+1}_{\text{scl}}}).
\end{equation*}
Choosing $\varepsilon$ small enough we may absorb the last term to the left hand side. Since $P_{0,\varphi_{\varepsilon}} + hV = e^{\varphi^2/2\varepsilon} e^{\varphi/h}(P_0(hD) + hV) e^{-\varphi/h} e^{-\varphi^2/2\varepsilon}$, we have the desired estimate.
\end{proof}

By using duality and the Hahn-Banach theorem in a standard way, one can convert this estimate (the case $s=-1$) into a solvability result.

\begin{prop} \label{prop:inhomogeneous_eq_l2}
Let $V \in L^{\infty}(\Omega)^{N \times N}$ and let $\varphi$ be a limiting Carleman weight for the Laplacian. If $h$ is small, then for any $f \in L^2(\Omega)^N$ there is a solution $u \in H^1(\Omega)^N$ of the equation 
\begin{equation} \label{inhomogeneous_eq}
e^{\varphi/h}(P_0(hD) + hV) e^{-\varphi/h} u = f \quad \text{in } \Omega,
\end{equation}
which satisfies $h \norm{u}_{H^1_{\text{scl}}} \lesssim \norm{f}$.
\end{prop}

In \cite{ksu} the partial data results for inverse problems were based on Carleman estimates with boundary terms, that is, estimates proved for functions whose support may extend up to the boundary. The above proof of the Carleman estimate for Dirac involved symbol calculus and estimates shifted to a different Sobolev index, which are problematic when applied to functions which are not compactly supported.

We will present another proof of a Carleman estimate for Dirac, which involves only integration by parts and gives an estimate with boundary terms. The estimate is weaker than the earlier ones, and works for $V = 0$. Following \cite{ksu}, we try to prove the estimate with boundary terms by writing the conjugated Dirac $P_{0,\varphi} = e^{\varphi/h} P_0(hD) e^{-\varphi/h}$ as $A+iB$ where $A$ and $B$ are self-adjoint, and by computing 
\begin{align*}
\norm{P_{0,\varphi} u}^2 &= ((A+iB)u|(A+iB)u) \\
 &= \norm{Au}^2 + \norm{Bu}^2 + (i[A,B]u|u) + \text{boundary terms}.
\end{align*}
In \cite{ksu}, which considered scalar operators, it was essential that the principal symbol $i(ab-ba)$ of $i[A,B]$ vanishes. Here however $a = P_0(\xi)$ and $b = P_0(\nabla \varphi)$ are matrices, and the principal symbol does not vanish. A part of it does vanish however: if $v = \nabla \varphi/\abs{\nabla \varphi}$ and $a = a_{\parallel} + a_{\perp}$, where $a_{\parallel} = P_0(\xi_{\parallel})$ with $\xi_{\parallel} = (\xi \cdot v)v$, then 
\begin{equation*}
a_{\parallel} b - b a_{\parallel} = (\xi \cdot \nabla \varphi) (P_0(v)^2 - P_0(v)^2) = 0.
\end{equation*}
This idea motivates the following proof. Here let $(u|v)_{\partial \Omega} = \int_{\partial \Omega} u \cdot \bar{v} \,dS$ and $\partial \Omega_{\pm} = \{ x \in \partial \Omega \,;\, \pm \nabla \varphi(x) \cdot \nu(x) > 0\}$.

\begin{prop} \label{prop:carleman_linear}
Let $\varphi(x) = \alpha \cdot x$ where $\alpha \in \mR^n$, $\abs{\alpha} = 1$. Let $h$ be sufficiently small. Then for any $u \in C^{\infty}(\closure{\Omega})^N$, one has the Carleman estimate 
\begin{multline*}
- h((\partial_{\nu} \varphi)u|u)_{\partial \Omega_{-}} \leq \norm{e^{\varphi/h} P_0(hD) e^{-\varphi/h} u}^2 + h((\partial_{\nu} \varphi)u|u)_{\partial \Omega_{+}} \\
 - ih (A_{\parallel} u|P_0(\nu_{\perp}) u)_{\partial \Omega} - ih (A_{\perp} u|P_0(\nu_{\parallel}) u)_{\partial \Omega},
\end{multline*}
where $A_{\parallel} = P_0(\alpha) \alpha \cdot hD$, $\nu_{\parallel} = (\alpha \cdot \nu)\alpha$, and $A_{\perp} = P_0(hD)-A_{\parallel}$, $\nu_{\perp} = \nu - \nu_{\parallel}$.
\end{prop}
\begin{proof}
Let first $\varphi \in C^{\infty}(\closure{\Omega} ; \mR)$ and $\nabla \varphi \neq 0$ on $\closure{\Omega}$. Let $P_{0,\varphi} = A + iB$ be as above, so that $A = P_0(hD)$ and $B = P_0(\nabla \varphi)$. Let $v = \nabla \varphi/\abs{\nabla \varphi} = (v_1, \ldots, v_n)$. We decompose $A$ as $A = A_{\parallel} + A_{\perp}$, where
\begin{equation*}
A_{\parallel} = \sum_{j=1}^n P_0(v)v_j hD_j, \quad A_{\perp} = P_0(hD) - \sum_{j=1}^n P_0(v)v_j hD_j.
\end{equation*}
It holds that 
\begin{multline*}
\norm{P_{0,\varphi} u}^2 = ((A+iB)u|(A+iB)u) = \norm{A_{\parallel} u}^2 + \norm{(A_{\perp}+iB)u}^2 \\
 + (A_{\parallel} u|A_{\perp} u) + (A_{\perp} u|A_{\parallel} u) + i(Bu|A_{\parallel} u) - i(A_{\parallel} u|Bu).
\end{multline*}
Use $P_0(v)^2 = I$ and integrate by parts to obtain 
\begin{align*}
 & i(Bu|A_{\parallel} u) - i(A_{\parallel} u|Bu) \\
 &= \sum_{j=1}^n [ i(P_0(\nabla \varphi)u|P_0(v) v_j hD_j u) - i(P_0(v) v_j hD_j u|P_0(\nabla \varphi)u) ] \\
 &= \sum_{j=1}^n [ i(u|\varphi_{x_j} hD_j u) - i(\varphi_{x_j} hD_j u|u) ] \\
 &= -h((\partial_{\nu} \varphi)u|u)_{\partial \Omega} + h((\Delta \varphi)u|u).
\end{align*}
For the terms involving $A_{\parallel}$ and $A_{\perp}$, one has 
\begin{align*}
 & (A_{\parallel} u|A_{\perp} u) + (A_{\perp} u|A_{\parallel} u) \\
 &= ((A_{\perp}^* A_{\parallel} + A_{\parallel}^* A_{\perp})u|u) + ih (A_{\parallel} u|P_0(\nu_{\perp}) u)_{\partial \Omega} + ih (A_{\perp} u|P_0(\nu_{\parallel}) u)_{\partial \Omega},
\end{align*}
where $\nu_{\parallel} = (\nu \cdot v)v$ and $\nu_{\perp} = \nu - \nu_{\parallel}$. So far we have not used any special properties of $\varphi$.

Now assume that $\varphi(x) = \alpha \cdot x$ is the linear weight and choose coordinates so that $\varphi(x) = x_1$. Then $v = e_1$ is a constant vector, and 
\begin{equation*}
A_{\parallel} = P_0(e_1)hD_1, \quad A_{\perp} = \sum_{j=2}^n P_0(e_j)hD_j.
\end{equation*}
We have $A_{\parallel}^* = A_{\parallel}$ and $A_{\perp}^* = A_{\perp}$, and 
\begin{equation*}
A_{\perp} A_{\parallel} + A_{\parallel} A_{\perp} = \sum_{j=2}^n (P_0(e_j)P_0(e_1) + P_0(e_1)P_0(e_j)) hD_1 hD_j.
\end{equation*}
This vanishes by \eqref{pzero_zeta_xi}.

From the above, we have for $u \in C^{\infty}(\closure{\Omega})^N$ the Carleman estimate 
\begin{multline*}
\norm{P_{0,\varphi} u}^2 = \norm{A_{\parallel} u}^2 + \norm{(A_{\perp}+iB)u}^2 \\
 - h((\partial_{\nu} \varphi)u|u)_{\partial \Omega} + ih (A_{\parallel} u|P_0(\nu_{\perp}) u)_{\partial \Omega} + ih (A_{\perp} u|P_0(\nu_{\parallel}) u)_{\partial \Omega}.
\end{multline*}
The desired estimate follows.
\end{proof}

\begin{remark}
The above proof can be carried out for the convexified linear weight $\varphi_{\varepsilon} = \varphi + \frac{h}{\varepsilon} \frac{\varphi^2}{2}$ where $\varphi(x) = \alpha \cdot x$, since $v$ is a constant vector also in that case. In this way one can include a $L^{\infty}$ potential in the Carleman estimate for $\varphi_{\varepsilon}$, as well as the term $h^2 \norm{u}^2$ on the left. However, due to the boundary terms involving $A_{\parallel}$ and $A_{\perp}$, it is not clear how to go back from $\varphi_{\varepsilon}$ to $\varphi$ in this case.
\end{remark}

\begin{remark}
In the proof of Proposition \ref{prop:carleman_linear}, the quantity $\Delta \varphi$ appears and one might expect the condition $\Delta \varphi = 0$ to be related to limiting Carleman weights for Dirac. This can also be seen by writing $P_{0,\varphi} = e^{\varphi/h} P_0(hD) e^{-\varphi/h}$ as $A + iB$ where $A = P_0(hD)$ and $B = P_0(\nabla \varphi)$, and by noting that 
\begin{equation*}
\norm{P_{0,\varphi} u}^2 = ((A-iB)(A+iB)u|u)
\end{equation*}
for test functions $u$, where $(A-iB)(A+iB)$ has full symbol 
\begin{equation*}
(a-ib)(a+ib) + \frac{h}{2i}\{a-ib,a+ib\}.
\end{equation*}
Here $\{c,d\}$ is the matrix symbol whose $(j,k)$th element is $\sum_{l=1}^n \{c_{jl}, d_{lk}\}$. A computation shows that $\frac{h}{2i}\{a-ib,a+ib\} = h (\Delta \varphi) I_N$, and thus the symbol of $(A-iB)(A+iB)$ is nonnegative definite if $\Delta \varphi = 0$. It is not clear to us if one can exploit harmonicity in proving Carleman estimates for Dirac in $\mR^n$. There are interesting recent results related to this approach and the semiclassical Fefferman-Phong inequality for systems in the work \cite{parmeggiani}.
\end{remark}

\section{CGO solutions, smooth case}

In this section we wish to construct complex geometrical optics solutions to $\mathcal{L}_V u = 0$, where $\mathcal{L}_V$ is given in \eqref{lv_definition}, \eqref{v_definition}. It will be convenient to consider $4 \times 4$ matrix solutions (that is, every column of the matrix is a solution) of the form 
\begin{equation} \label{wkb_ansatz}
U = e^{-\rho/h}(C_0 + h C_1 + \ldots + h^{N-1} C_{N-1} + h^{N-1} R_N).
\end{equation}
This is a WKB ansatz for the solution, where $\rho = \varphi + i\psi$ is a complex phase function with $\varphi$ a Carleman weight, $C_j$ are matrices which correspond to amplitudes, and $R_N$ is a correction term. These are called complex geometrical optics solutions because the phase is complex.

We will give a construction for smooth coefficients (Proposition \ref{prop:cgo_construction_smooth}), in which case there is an arbitrarily long asymptotic expansion of the form \eqref{wkb_ansatz} where the successive terms have increasing decay in $h$. This can be achieved if $\varphi$ is a limiting Carleman weight for the Laplacian. In the next section we consider Lipschitz coefficients (Proposition \ref{prop:cgo_construction_nonsmooth}), in which case $\varphi$ is the linear weight for simplicity. Then the successive terms in \eqref{wkb_ansatz} will have only a limited decay in $h$, but we compute sufficiently many terms to be able to prove a uniqueness result for the inverse problem.

Suppose that $A, q_{\pm} \in C^{\infty}(\closure{\Omega})$. Writing $P_{\rho} = e^{\rho/h} (P_0(hD)+hV) e^{-\rho/h} = iP_0(\nabla \rho) + h \mathcal{L}_V$, inserting the ansatz \eqref{wkb_ansatz} in the equation $\mathcal{L}_V U = 0$, and collecting like powers of $h$, results in the equations 
\begin{align*}
iP_0(\nabla \rho) C_0 &= 0,  \\
iP_0(\nabla \rho) C_1 &= - \mathcal{L}_V C_0,  \\
 & \vdots \\
iP_0(\nabla \rho) C_{N-1} &= - \mathcal{L}_V C_{N-2}, \\
P_{\rho} R_N &= - h \mathcal{L}_V C_{N-1}.
\end{align*}

We now give a procedure for solving these equations. The first equation implies that the kernel of $P_0(\nabla \rho)$ should be nontrivial. The same applies to the kernel of $P_0(\nabla \rho)^2 = (\nabla \rho)^2 I_4$, which gives the condition 
\begin{equation} \label{eikonal_eq}
(\nabla \rho)^2 = 0.
\end{equation}
This is an eikonal equation for the complex phase. If $\varphi$ is given, the equations for $\psi$ become 
\begin{equation} \label{eikonal_eq_psi}
\abs{\nabla \psi}^2 = \abs{\nabla \varphi}^2, \quad \nabla \varphi \cdot \nabla \psi = 0.
\end{equation}
Assume for the moment that \eqref{eikonal_eq_psi} is solvable. From \eqref{eikonal_eq} we get $\text{ker}\,P_0(\nabla \rho) = \text{im}\,P_0(\nabla \rho)$, since the image is contained in the kernel and $\text{rank}(P_0(\nabla \rho)) = 2$. We choose $C_0 = P_0(\nabla \rho) \widetilde{C}_0$ where $\widetilde{C}_0$ is a smooth matrix to be determined.

Moving on to the second equation, we use the commutator identity 
\begin{equation*}
\mathcal{L}_V P_0(\nabla \rho) = M_{A} + P_0(\nabla \rho)(-P_0(D+A) + Q_I),
\end{equation*}
where $Q_I = \left( \begin{smallmatrix} q_- I_2 & 0 \\ 0 & q_+ I_2 \end{smallmatrix} \right)$ and $M_A$ is the transport operator 
\begin{equation*}
M_{A} = (2(\nabla \rho \cdot (D + A)) + \frac{1}{i} \Delta \rho) I_4.
\end{equation*}
The equation for $C_1$ then reads 
\begin{equation*}
P_0(\nabla \rho) C_1 = P_0(\nabla \rho) \frac{1}{i} (P_0(D+A)-Q_I) \widetilde{C}_0 + i M_A \widetilde{C}_0.
\end{equation*}
A solution is given by $C_1 = \frac{1}{i} (P_0(D+A)-Q_I) \widetilde{C}_0 + P_0(\nabla \rho) \widetilde{C}_1$, for some $\widetilde{C}_1$ to be determined, provided that 
\begin{equation} \label{transport_eq}
M_A \widetilde{C}_0 = 0.
\end{equation}
Under certain conditions which are stated below, the transport equation \eqref{transport_eq} has a smooth solution $\widetilde{C}_0$.

For the third equation we use that 
\begin{equation} \label{secondorder_identity}
(P_0(D+A) + Q)(P_0(D+A) - Q_I) = H_{A,W},
\end{equation}
where $H_{A,W} = (D+A)^2 I_4 + W$ is the magnetic Schr\"odinger operator, with 
\begin{equation*}
W = \left( \begin{array}{cc} \sigma \cdot (\nabla \times A) - q_+ q_- I_2 & -\sigma \cdot Dq_+ \\ -\sigma \cdot Dq_- & \sigma \cdot (\nabla \times A) - q_+ q_- I_2 \end{array} \right).
\end{equation*}
Then the equation for $C_2$ becomes 
\begin{equation*}
P_0(\nabla \rho) C_2 = H_{A,W} \widetilde{C}_0 + P_0(\nabla \rho) \frac{1}{i}(P_0(D+A) - Q_I) \widetilde{C}_1 + i M_A \widetilde{C}_1.
\end{equation*}
Choosing $\widetilde{C}_1$ as a solution of the transport equation $M_A \widetilde{C}_1 = i H_{A,W} \widetilde{C}_0$, we obtain $C_2$ as 
\begin{equation*}
C_2 = \frac{1}{i}(P_0(D+A) - Q_I) \widetilde{C}_1 + P_0(\nabla \rho) \widetilde{C}_2.
\end{equation*}
Continuing in this way we obtain smooth matrices $C_3$, \ldots, $C_{N-1}$. The equation for $h^{N-1} R_N$ may be solved by Proposition \ref{prop:inhomogeneous_eq_l2}, which ends the construction of solutions.

We still need to consider the solvability of the eikonal equation \eqref{eikonal_eq_psi} and transport equation \eqref{transport_eq}. As discussed in \cite{ksu}, these equations can be solved provided that $\varphi$ is a limiting Carleman weight for the Laplacian, under some geometric assumptions. Suppose that $\Omega \subset \subset \tilde{\Omega}$ and that $\varphi$ is a limiting Carleman weight in $\tilde{\Omega}$. Assume that 
\begin{quote}
$\Omega$ is contained in the union of integral curves of $\nabla \varphi$, all passing through the smooth hypersurface $G = \varphi^{-1}(C_0)$ in $\tilde{\Omega}$.
\end{quote}
Then one may solve the eikonal equation $\abs{\nabla \psi_0}^2 = \abs{\nabla \varphi}^2$ on $G$ by letting $\psi_0$ be the distance in the metric $\abs{\nabla \varphi}^2 e_0$ on $G$ to a point or hypersurface ($e_0$ is the induced Euclidean metric), chosen so that $\psi_0$ is smooth. The limiting Carleman condition implies that $\psi$, obtained from $\psi_0$ by extending it as constant along integral curves of $\nabla \varphi$, will solve the eikonal equation \eqref{eikonal_eq_psi}.

Given $\psi$ satisfying \eqref{eikonal_eq_psi}, one has $[\nabla \varphi, \nabla \psi] = c\nabla \varphi + d\nabla \psi$ by \cite{ksu} (see also \cite[Lemma 3.1]{knudsensalo}). It follows by \cite{duistermaathormander} that transport equations such as \eqref{transport_eq} are solvable, provided that the leaves of the foliation generated by $\nabla \varphi$ and $\nabla \psi$ are not contained in $\closure{\Omega}$.

If $\varphi$ is the linear weight or the logarithmic weight $\log\,\abs{x-x_0}$, with $x_0$ outside the convex hull of $\closure{\Omega}$, then these geometric conditions are satisfied and the equations \eqref{eikonal_eq_psi} and \eqref{transport_eq} may be solved explicitly (for the logarithmic weight see \cite{dksu}). Thus, the following result holds in particular for these choices of the weight $\varphi$.

\begin{prop} \label{prop:cgo_construction_smooth}
Let $\varphi$ be a limiting Carleman weight for the Laplacian in $\tilde{\Omega}$, where $\Omega \subset \subset \tilde{\Omega}$ and where the above conditions for solving \eqref{eikonal_eq_psi} and \eqref{transport_eq} are satisfied. If $A, q_{\pm} \in C^{\infty}(\closure{\Omega})$, then the equation $\mathcal{L}_V U = 0$ has a solution of the form \eqref{wkb_ansatz}, where 
\begin{eqnarray*}
 & C_0 = P_0(\nabla \rho) \widetilde{C}_0, \quad M_A \widetilde{C}_0 = 0, & \\
 & C_1 = \frac{1}{i} (P_0(D+A)-Q_I) \widetilde{C}_0 + P_0(\nabla \rho) \widetilde{C}_1, \quad M_A \widetilde{C}_1 = i H_{A,W} \widetilde{C}_0, \\
 & \vdots & \\
 & C_{N-1} = \frac{1}{i} (P_0(D+A)-Q_I) \widetilde{C}_{N-2} + P_0(\nabla \rho) \widetilde{C}_{N-1}, \quad \widetilde{C}_{N-1} \text{ smooth}, \\
 & \norm{C_j}_{W^{1,\infty}(\Omega)} \lesssim 1, \quad \norm{R_N} + h \norm{\nabla R_N} \lesssim 1. &
\end{eqnarray*}
\end{prop}

\section{CGO solutions, Lipschitz case}

Here we construct solutions in the case of compactly supported Lipschitz continuous coefficients, $A, q_{\pm} \in W^{1,\infty}_c(\Omega)$. The construction will be carried out for the linear phase $\rho(x) = \zeta \cdot x$ where $\zeta \in \mC^3$, $\zeta^2 = 0$, and $\abs{\re\,\zeta} = \abs{\im\,\zeta}= 1$.

To deal with nonsmooth coefficients, we introduce mollifiers $\eta_{\varepsilon}(x) = \varepsilon^{-3} \eta(x/\varepsilon)$ where $\eta \in C^{\infty}_c(\mR^3)$ is supported in the unit ball and $0 \leq \eta \leq 1$, $\int \eta \,dx = 1$. Decompose $A = A^{\sharp} + A^{\flat}$ and $Q = Q^{\sharp} + Q^{\flat}$, where $A^{\sharp} = A \ast \eta_{\varepsilon}$ and $Q^{\sharp} = Q \ast \eta_{\varepsilon}$, with the special choice 
\begin{equation*}
\varepsilon = h^{\sigma}
\end{equation*}
where $\sigma > 0$ is small. Also write $V = V^{\sharp} + V^{\flat}$ where $V^{\sharp} = P_0(A^{\sharp}) + Q^{\sharp}$. We find a solution to $\mathcal{L}_{V} U = 0$ of the form 
\begin{equation*}
U = e^{-\rho/h}(C_0 + h C_1 + R),
\end{equation*}
where 
\begin{align*}
iP_0(\zeta) C_0 &= 0, \\
iP_0(\zeta) C_1 &= - \mathcal{L}_{V^{\sharp}} C_0, \\
P_{\rho} R &= - h^2 \mathcal{L}_{V^{\sharp}} C_1 - h V^{\flat}(C_0 + h C_1).
\end{align*}
We make the choices 
\begin{align*}
C_0 &= P_0(\zeta) e^{i\phi^{\sharp}}, \\
C_1 &= \frac{1}{i} (P_0(D+A^{\sharp}) - Q_I^{\sharp}) e^{i\phi^{\sharp}} I_4,
\end{align*}
where $\phi^{\sharp}$ is the solution to $\zeta \cdot (\nabla \phi^{\sharp} + A^{\sharp}) = 0$ given by 
\begin{equation*}
\phi^{\sharp} = (\zeta \cdot \nabla)^{-1}(-\zeta \cdot A^{\sharp}).
\end{equation*}
Here $(\zeta \cdot \nabla)^{-1}$ is the Cauchy transform 
\begin{equation*}
(\zeta \cdot \nabla)^{-1} f(x) = \frac{1}{2\pi} \int_{\mR^2} \frac{1}{y_1+iy_2} f(x-y_1 \re\,\zeta - y_2 \im\,\zeta) \,dy_1 \,dy_2.
\end{equation*}
The following result will be used to solve for the error term $R$. We will need the additional small parameter $\tilde{h}$ to prove some extra decay for $R$.

\begin{prop} \label{prop:inhomogeneous_eq_h1_tilde}
Let $\Omega \subseteq \mR^n$ be a bounded open set, let $P_0$ be a Dirac operator in $\mR^n$, and assume that $V \in W^{1,n} \cap L^{\infty}(\Omega)^{N \times N}$. Also let $\zeta \in \mC^n$, $\zeta^2 = 0$. If $h$, $\tilde{h}$ are small, then for any $f \in H^1(\Omega)^N$ there is a solution $u \in H^1(\Omega)^N$ of the equation 
\begin{equation*}
(P_0(hD+\zeta)+hV) u = f \qquad \text{in $\Omega$},
\end{equation*}
which satisfies 
\begin{equation*}
h (\norm{u} + \tilde{h} \norm{\nabla u}) \lesssim \norm{f} + \tilde{h} \norm{\nabla f}.
\end{equation*}
\end{prop}
\begin{proof}
We wish to show the Carleman estimate 
\begin{equation} \label{carleman_hminus1}
h \norm{\langle \tilde{h}D \rangle^{-1} u}  \leq C \norm{\langle \tilde{h}D \rangle^{-1} (P_0(hD+\zeta) + hV) u}.
\end{equation}
The result follows from this in a standard way by the Hahn-Banach theorem.

Let $\rho(x) = \zeta \cdot x$ and $\rho_{\varepsilon} = \rho + \frac{h}{\varepsilon} \frac{\rho^2}{2}$. Let $u \in C^{\infty}_c(\Omega)$, and choose $\chi \in C^{\infty}_c(\tilde{\Omega})$ where $\chi = 1$ near $\Omega \subset \subset \tilde{\Omega}$. Then the pseudolocal estimate 
\begin{equation*}
\norm{(1-\chi) \langle \tilde{h}D \rangle^{-1} v} \leq C_M \tilde{h}^M \norm{\langle \tilde{h}D \rangle^{-1} v}, \quad v \in C_c^{\infty}(\Omega), \text{ any $M$},
\end{equation*}
and the Carleman estimate \eqref{carleman_freedirac_convexified} imply that when $h$, $\tilde{h}$ are small, 
\begin{align*}
h \norm{\langle \tilde{h}D \rangle^{-1} u} &\leq h \norm{\chi \langle \tilde{h}D \rangle^{-1} u} + h \norm{(1-\chi) \langle \tilde{h}D \rangle^{-1} u} \\
 &\leq C\sqrt{\varepsilon} \norm{P_0(hD+\nabla \rho_{\varepsilon}) (\chi \langle \tilde{h}D \rangle^{-1} u)} + \frac{h}{2} \norm{\langle \tilde{h}D \rangle^{-1} u} \\
 &\leq C\sqrt{\varepsilon} \norm{\chi P_0(hD+\nabla \rho_{\varepsilon}) (\langle \tilde{h}D \rangle^{-1} u)} + \frac{3h}{4} \norm{\langle \tilde{h}D \rangle^{-1} u}.
\end{align*}
Since $\nabla \rho_{\varepsilon} = (1+\frac{h}{\varepsilon} \zeta \cdot x)\zeta$, we obtain 
\begin{align*}
h \norm{\langle \tilde{h}D \rangle^{-1} u} &\leq C\sqrt{\varepsilon} \norm{\chi \langle \tilde{h}D \rangle^{-1} P_0(hD+\nabla \rho_{\varepsilon}) u} + \frac{C h}{\sqrt{\varepsilon}} \norm{\chi P_0(\zeta) [\zeta \cdot x, \langle \tilde{h}D \rangle^{-1}] u}.
\end{align*}
Here $[\zeta \cdot x, \langle \tilde{h}D \rangle^{-1}] = \tilde{h} R \langle \tilde{h}D \rangle^{-1}$ where $R$ is bounded on $L^2$ with norm $\lesssim 1$. If $\tilde{h}$ is small enough, we obtain 
\begin{equation*}
h \norm{\langle \tilde{h}D \rangle^{-1} u}  \leq  C\sqrt{\varepsilon} \norm{\langle \tilde{h}D \rangle^{-1} P_0(hD+\nabla \rho_{\varepsilon}) u}.
\end{equation*}
If $\varepsilon$ is small enough, the estimate remains true with $P_0(hD+\nabla \rho_{\varepsilon})$ replaced by $P_0(hD+\nabla \rho_{\varepsilon}) + hV$, since 
\begin{equation*}
\norm{\langle \tilde{h}D \rangle^{-1} V u} \leq C \norm{\langle \tilde{h}D \rangle^{-1} u}.
\end{equation*}
Then \eqref{carleman_hminus1} follows since $e^{\rho_{\varepsilon}/h} = m e^{\rho/h}$ with $\norm{m}_{W^{1,\infty}(\Omega)}$ bounded.
\end{proof}

Proposition \ref{prop:inhomogeneous_eq_h1_tilde}, and the fact that $V \in W^{1,\infty}$, gives a correction term $R$ satisfying 
\begin{multline*}
\norm{R} + \tilde{h} \norm{\nabla R} \lesssim \norm{h \mathcal{L}_V^{\sharp} C_1 + V^{\flat}(C_0 + h C_1)} + \tilde{h} \norm{\nabla (h \mathcal{L}_V^{\sharp} C_1 + V^{\flat}(C_0 + h C_1))} \\
 \lesssim h^{1-\sigma} + \norm{V^{\flat}}_{L^2} + h^{1-2\sigma} \tilde{h} + \tilde{h} \norm{\nabla V^{\flat}}_{L^2} + \tilde{h} \norm{V^{\flat}}_{L^2}.
\end{multline*}
Choosing $\tilde{h} = h^{\sigma_1}$ for $\sigma_1 > 0$ small, one has 
\begin{equation*}
\norm{R} = o(1), \quad \norm{\nabla R} = o(1)
\end{equation*}
as $h \to 0$. Thus we obtain a solution 
\begin{equation*}
U = e^{-\frac{\zeta \cdot x}{h}}(P_0(\zeta) e^{i\phi^{\sharp}} + \frac{h}{i} (P_0(D+A^{\sharp}) - Q_I^{\sharp}) e^{i\phi^{\sharp}} I_4 + R),
\end{equation*}
where $\norm{R} = o(1)$. Also, since $P_{\rho} = iP_0(\nabla \rho) + h \mathcal{L}_V$, we have 
\begin{equation*}
\norm{P_0(\zeta) R} \lesssim h \norm{R}_{H^1} + h^2 \norm{\mathcal{L}_{V^{\sharp}} C_1} + h \norm{C_0 + h C_1}_{L^{\infty}} \norm{V^{\flat}} = o(h).
\end{equation*}
Noting that by \eqref{transport_eq} we may replace $e^{i \phi^{\sharp}}$ by $e^{i \phi^{\sharp} + ik \cdot x}$ where $k \cdot \zeta = 0$, we have arrived at the solutions for Lipschitz coefficients.

\begin{prop} \label{prop:cgo_construction_nonsmooth}
Let $A, q_{\pm} \in W^{1,\infty}_c(\Omega)$, and let $\zeta \in \mC^3$ satisfy $\zeta^2 = 0$ and $\abs{\re\,\zeta} = \abs{\im\,\zeta}= 1$. There exists a solution $U \in H^1(\Omega)^{4 \times 4}$ to $\mathcal{L}_V U = 0$ in $\Omega$, of the form 
\begin{equation*}
U = e^{-\frac{1}{h} \zeta \cdot x} e^{i\phi^{\sharp} + ik \cdot x} (P_0(\zeta) + \frac{h}{i} (P_0(\nabla \phi^{\sharp} + k + A^{\sharp}) - Q_I^{\sharp}) + R),
\end{equation*}
where $k \in \mR^3$ with $k \cdot \zeta = 0$, and where $\norm{R} = o(1), \norm{P_0(\zeta) R} = o(h)$ as $h \to 0$.
\end{prop}

\section{Uniqueness result}

We will prove Theorem \ref{thm:uniqueness}. The first step is a standard reduction to a larger domain.

\begin{lemma} \label{lemma:integral_identity}
Let $\Omega \subset \subset \Omega'$ be two bounded open sets in $\mR^3$, and let $A_j, q_{\pm,j} \in W^{1,\infty}(\Omega')$ satisfy $A_1 = A_2$ and $q_{\pm,1} = q_{\pm,2}$ in $\Omega' \smallsetminus \Omega$. If $C_{V_1} = C_{V_2}$ in $\Omega$, then $C_{V_1} = C_{V_2}$ in $\Omega'$ and 
\begin{equation} \label{integral_identity}
\int_{\Omega'} U_2^* (V_1-V_2)U_1 \,dx = 0
\end{equation}
for any solutions $U_j \in H^1(\Omega')^{4 \times 4}$ of $\mathcal{L}_{V_j} U_j = 0$ in $\Omega'$.
\end{lemma}
\begin{proof}
If $\mathcal{L}_{V_1} u' = 0$ in $\Omega'$, then there is $v \in \mathcal{H}(\Omega)^2$ such that $\mathcal{L}_{V_2} v = 0$ in $\Omega$ and $v_{\pm}|_{\partial \Omega} = u'_{\pm}|_{\partial \Omega}$. Let $v' = v$ in $\Omega$ and $v' = u'$ in $\Omega' \smallsetminus \Omega$. It is easy to see that $v' \in \mathcal{H}(\Omega')^2$ and $\mathcal{L}_{V_2} v' = 0$ in $\Omega'$, showing that $C_{V_1} \subseteq C_{V_2}$ in $\Omega'$. The same argument in the other direction gives $C_{V_1} = C_{V_2}$ in $\Omega'$.

Let $U_j$ be as described. Writing $(U|V) = \int_{\Omega'} V^* U \,dx$, we have  
\begin{equation*}
((V_1-V_2)U_1|U_2) = -(P_0(D)U_1|U_2) + (U_1|P_0(D)U_2).
\end{equation*}
Since $C_{V_1} = C_{V_2}$, there is $\tilde{U}_2 \in \mathcal{H}(\Omega')^{4 \times 4}$ with $\mathcal{L}_{V_2} \tilde{U}_2 = 0$ in $\Omega'$ and also $(U_1-\tilde{U}_2)_{\pm} \in H^1_0(\Omega')^{2 \times 4}$. Thus, writing $U_1 = (U_1-\tilde{U}_2) + \tilde{U}_2$ and integrating by parts, we obtain 
\begin{align*}
((V_1-V_2)U_1|U_2) &= -(P_0(D)\tilde{U}_2|U_2) + (\tilde{U}_2|P_0(D)U_2) \\
 &= (V_2 \tilde{U}_2|U_2) - (\tilde{U}_2|V_2 U_2) = 0.
\end{align*}
\end{proof}

Assume the conditions of Theorem \ref{thm:uniqueness}. By a gauge transformation we may assume that the normal components of $A_j$ vanish on $\partial \Omega$. Then by Theorem \ref{thm:boundary} we know that $A_1 = A_2$ and $q_{\pm,1} = q_{\pm,2}$ on $\partial \Omega$. Let $\Omega'$ be a ball such that $\Omega \subset \subset \Omega'$, and extend $A_j$ and $q_{\pm,j}$ as compactly supported Lipschitz functions in $\Omega'$ so that $A_1 = A_2$ and $q_{\pm,1} = q_{\pm,2}$ outside $\Omega$.

Lemma \ref{lemma:integral_identity} shows that in the proof of Theorem \ref{thm:uniqueness}, we may assume that $\Omega$ is a ball, the coefficients are in $W^{1,\infty}_c(\Omega)$, $C_{V_1} = C_{V_2}$ in $\Omega$, and \eqref{integral_identity} holds for solutions in $\Omega$. The recovery of the coefficients proceeds similarly as in \cite{nakamuratsuchida}, using now the solutions provided by Proposition \ref{prop:cgo_construction_nonsmooth}. We will give the details since one needs to ensure that the estimates for nonsmooth solutions are sufficient for this argument. Theorem \ref{thm:uniqueness} will follow from the two propositions below.

\begin{prop}
$\nabla \times A_1 = \nabla \times A_2$ in $\Omega$.
\end{prop}
\begin{proof}
Choose $\zeta \in \mC^3$ with $\zeta^2 = 0$ and $\abs{\re\,\zeta} = \abs{\im\,\zeta} = 1$. Let $k \in \mR^3$ be orthogonal to $\re\,\zeta$ and $\im\,\zeta$. By Proposition \ref{prop:cgo_construction_nonsmooth} we may choose solutions to $\mathcal{L}_{V_j} U_j = 0$ in $\Omega$ of the form 
\begin{eqnarray*}
 & U_1 = e^{\frac{1}{h} \zeta \cdot x} e^{ik \cdot x} (-P_0(\zeta) e^{i\phi_1^{\sharp}} + R_1), & \\
 & U_2 = e^{-\frac{1}{h} \bar{\zeta} \cdot x} (P_0(\bar{\zeta}) e^{i\bar{\phi_2^{\sharp}}} + R_2), & 
\end{eqnarray*}
where $\phi_j^{\sharp} = (\zeta \cdot \nabla)^{-1}(-\zeta \cdot A_j^{\sharp})$ and $\norm{R_j} = o(1)$ as $h \to 0$. Then 
\begin{equation*}
U_2^* = e^{-\frac{1}{h} \zeta \cdot x} (P_0(\zeta) e^{-i\phi_2^{\sharp}} + R_2^*).
\end{equation*}
Inserting these in \eqref{integral_identity} and letting $h \to 0$ gives 
\begin{equation*}
\int e^{ik \cdot x} e^{i\phi} P_0(\zeta) P_0(A_1-A_2) P_0(\zeta) \,dx = 0,
\end{equation*}
where $\phi = (\zeta \cdot \nabla)^{-1}(-\zeta \cdot (A_1-A_2))$ and we have used $P_0(\zeta) Q_j P_0(\zeta) = 0$. Using the commutator identity for $P_0$ one obtains 
\begin{equation*}
\int e^{ik \cdot x} e^{i\phi} (\zeta \cdot (A_1-A_2)) \,dx = 0.
\end{equation*}
Lemma 6.2 in \cite{saloreconstruction}, which is based on \cite{eskinralston}, implies that the same identity is true with $e^{i\phi}$ replaced by $1$. Consequently 
\begin{equation*}
\int e^{ik \cdot x} (\zeta \cdot (A_1-A_2)) \,dx = 0,
\end{equation*}
for all $k$ orthogonal to $\re\,\zeta$, $\im\,\zeta$. This implies the vanishing of the Fourier transform of components of $\nabla \times (A_1-A_2)$.
\end{proof}

\begin{prop}
$q_{\pm,1} = q_{\pm,2}$ in $\Omega$.
\end{prop}
\begin{proof}
Since $\nabla \times A_1 = \nabla \times A_2$ and $A_j \in W^{1,\infty}_c(\Omega)$ where $\Omega$ is a ball, one obtains $A_1 - A_2 = \nabla p$ for $p \in W^{2,\infty}(\Omega)$ where one may choose $p|_{\partial \Omega} = 0$. Thus, by a gauge transformation we may assume that $A_1 = A_2 = A$ where $A \in W^{1,\infty}_c(\Omega)$. Fix $k \in \mR^3$ and take $\zeta \in \mC^3$ with $\zeta^2 = 0$ and $k \cdot \zeta = 0$. By Proposition \ref{prop:cgo_construction_nonsmooth}, we take solutions to $\mathcal{L}_{V_j} U_j = 0$ of the form 
\begin{eqnarray*}
 & U_1 = e^{\frac{1}{h} \zeta \cdot x} e^{ik \cdot x} e^{i\phi^{\sharp}} (-P_0(\zeta) + \frac{h}{i} (P_0(\nabla \phi^{\sharp} + k + A^{\sharp}) - Q_{1,I}^{\sharp}) + R_1), & \\
 & U_2 = e^{-\frac{1}{h} \bar{\zeta} \cdot x} e^{i\bar{\phi^{\sharp}}} (P_0(\bar{\zeta}) + \frac{h}{i} (P_0(\nabla \bar{\phi^{\sharp}} + A^{\sharp}) - Q_{2,I}^{\sharp}) + R_2), & 
\end{eqnarray*}
where $\phi^{\sharp} = (\zeta \cdot \nabla)^{-1}(-\zeta \cdot A^{\sharp})$ and $\norm{R_j} = o(1)$, $\norm{P_0(\zeta) R_1} + \norm{P_0(\bar{\zeta}) R_2} = o(h)$ as $h \to 0$. Also, 
\begin{equation*}
U_2^* = e^{-\frac{1}{h} \zeta \cdot x} e^{-i\phi^{\sharp}} (P_0(\zeta) - \frac{h}{i} (P_0(\nabla \phi^{\sharp} + A^{\sharp}) - Q_{2,I}^{\sharp}) + R_2^*),
\end{equation*}
with $\norm{R_2^* P_0(\zeta)} = o(h)$. The identity \eqref{integral_identity} implies 
\begin{multline*}
\int e^{ik \cdot x} (P_0(\zeta) - \frac{h}{i} (P_0(\nabla \phi^{\sharp} + A^{\sharp}) - Q_{2,I}^{\sharp}) + R_2^*)(Q_1-Q_2) \\ 
 \cdot (-P_0(\zeta) + \frac{h}{i} (P_0(\nabla \phi^{\sharp} + k + A^{\sharp}) - Q_{1,I}^{\sharp}) + R_1) \,dx = 0.
\end{multline*}
The highest order term in $h$ is $P_0(\zeta) (Q_1-Q_2) P_0(\zeta)$, and this vanishes. Also the terms involving $P_0(\nabla \phi^{\sharp} + A^{\sharp})$ go away, since for $Q = Q_1-Q_2$, 
\begin{align*}
& P_0(\zeta) Q P_0(\nabla \phi^{\sharp} + A^{\sharp}) + P_0(\nabla \phi^{\sharp} + A^{\sharp}) Q P_0(\zeta) \\
&= Q_I(P_0(\zeta) P_0(\nabla \phi^{\sharp} + A^{\sharp}) + P_0(\nabla \phi^{\sharp} + A^{\sharp}) P_0(\zeta)) \\
&= 2 Q_I (\zeta \cdot (\nabla \phi^{\sharp} + A^{\sharp})) = 0.
\end{align*}
All the terms involving $R_j$ are $o(h)$. For the term involving $R_2^* Q R_1$, this is seen by using \eqref{pzero_zeta_xi} and writing 
\begin{equation*}
R_1 = \frac{1}{\sqrt{2}} (P_0(\zeta) P_0(\bar{\zeta}) R_1 + P_0(\bar{\zeta}) P_0(\zeta) R_1) = P_0(\zeta) o(1) + o(h).
\end{equation*}
Therefore, dividing the integral identity by $h$ and letting $h \to 0$ gives 
\begin{equation*}
\int e^{ik \cdot x} [ P_0(\zeta) Q (P_0(k) - Q_{1,I}) - Q_{2,I} Q P_0(\zeta) ] \,dx = 0.
\end{equation*} 
Commuting $P_0(\zeta)$ to the left gives 
\begin{equation*}
\int e^{ik \cdot x} P_0(\zeta) (Q P_0(k) - Q_1 Q_{1,I} + Q_2 Q_{2,I}) \,dx = 0.
\end{equation*}
This is true also when $\zeta$ is replaced by $\bar{\zeta}$, and adding the two identities and multiplying by $P_0(\re\,\zeta)$ on the left implies that 
\begin{equation*}
\int e^{ik \cdot x} ((Q_1-Q_2) P_0(k) - q_{+,1} q_{-,1} I_4 + q_{+,2} q_{-,2} I_4) \,dx = 0.
\end{equation*}
Looking at the off-diagonal $2 \times 2$ blocks shows 
\begin{equation*}
\int e^{ik \cdot x} (q_{\pm,1} - q_{\pm,2})(\sigma \cdot k) \,dx = 0.
\end{equation*}
The claim follows upon multiplying by $\sigma \cdot k$.
\end{proof}

\section{Boundary determination}

In this section we show that $q_{\pm}$ and the tangential component of $A$ at the boundary are uniquely determined by the Cauchy data set $C_{V}$. More precisely, we have the following.

\begin{prop}
\label{boundary determination}
Let $A_j$, $q_{\pm,j}$ be $W^{1,\infty}(\Omega)$ coefficients, $j = 1,2$, and let $\Omega$ have $C^1$ boundary. If $C_{V_1} = C_{V_2}$, then
\[(A_1 - A_2)(x_0)\cdot\hat t = 0\ \ \ \ \forall x_0\in\partial\Omega, \forall \hat t\in T_{x_0}(\partial\Omega),\]
\[q_{\pm,1}(x_0) = q_{\pm,2}(x_0)\ \ \ \forall x_0\in \partial\Omega.\]
\end{prop}

Following the idea of \cite{brownsalo} we first construct a sequence of solutions which concentrate at $x_0$ in the limit. We assume without loss of generality that $x_0 = 0$ and that $\Omega$ is defined by the function $\rho \in C^1(\mR^3 ; \mR)$ in such a way that $\Omega = \{ x \in \mR^3 \,;\, \rho(x) > 0\}$, $\partial \Omega = \{ x \in \mR^3 \,;\, \rho(x) = 0 \}$, and the outer unit normal of $\partial \Omega$ at $0$ is $-e_3= -\nabla\rho(0)$.

Let now $\eta(x)$ be a smooth function supported in $B(0,1/2)$ such that
\[\int_{\R^2}\eta(x',0)^2 \,dx' = 1.\]
For all $M>0$, define $\eta_M(x) := \eta(M(x',\rho(x)))$. Define 
\begin{equation*}
u_0(x) = e^{N(i\hat t\cdot x - \rho(x))}\eta_M(x),
\end{equation*}
where $\hat t \in T_{0}(\partial\Omega)$ is a unit vector and $N$ is chosen so that 
\begin{equation*}
N^{-1} = M^{-1}\omega(M^{-1}),
\end{equation*}
with $\omega(\cdot)$ a modulus of continuity for $\nabla \rho$. We prove the following.

\begin{prop}
\label{solutions concentration on boundary}
For $M$ large enough, one can find $H^1(\Omega)^{4\times 4}$ solutions to the equation
\[\mathcal{L}_V U = 0\]
of the form
\[U = NP_0(\hat t + i \nabla\rho) u_0 + R\]
with
\[\|R\|_{L^2}\leq CN^{-1/2}.\]
\end{prop}
To prove this result, we will use the following two lemmas from \cite{brownsalo}.

\begin{lemma}
\label{behaviour of u0}
Let $\eta$ be smooth and supported in $B(0,1/2)$. If $\eta_M(x) = \eta(M(x',\rho(x)))$ and $N^{-1} = M^{-1}\omega(M^{-1})$, we have 
\[\lim\limits_{M\to\infty} M^2 N \int_\Omega exp(-2N\rho(x))\eta_M(x) \,dx = 1/2\int_{\R^{2}}\eta(x',0) \,dx'\]
and
\[|\int_\Omega exp(-2N\rho(x))\eta_M(x) \,dx|\leq C M^{-2} N^{-1}.\]
\end{lemma}

\begin{lemma}
\label{second order equation}
Let $A:\closure{\Omega} \to \C^3$ be a continuous vector field and $k :\overline\Omega \to \C$ be a continuous function. If $0$ is not an eigenvalue of the operator 
\[-\Delta + A\cdot D + k : H^1_0(\Omega)\to H^{-1}(\Omega),\]
then there exist solutions to
\[(-\Delta + A\cdot D + k)u = 0\]
of the form
\[u = u_0 + \tilde u_1\]
with $\tilde u_1 \in H^1_0(\Omega)$ and $\|\tilde u_1\|_{H^1(\Omega)} \leq CN^{-1/2}$.
\end{lemma}

At this point we recall some matrix identities used already before. If $Q = \left( \begin{smallmatrix} q_+ I_2 & 0 \\ 0 & q_- I_2 \end{smallmatrix} \right)$ and $Q_I = \left( \begin{smallmatrix} q_- I_2 & 0 \\ 0 & q_+ I_2 \end{smallmatrix} \right)$ and if $a,b\in \mC^3$, we have 
\begin{eqnarray}
\label{matrix identities}
P_0(a)P_0(b) + P_0(b)P_0(a) = 2 a\cdot b,\ \ \ P_0(a) Q = Q_I P_0(a).
\end{eqnarray}

\begin{proof}
(of Proposition \ref{solutions concentration on boundary}) Recall from \eqref{secondorder_identity} that  
\begin{eqnarray}
\label{square of operator}
\mathcal{L}_V (P_0(D+A)-Q_I) = -\Delta + 2A\cdot D + \tilde{W}
\end{eqnarray}
where $\tilde{W}$ is a $4\times 4$ matrix function with $L^\infty$ entries. Choose $\lambda$ such that $0$ is not an eigenvalue of the scalar operator 
\[-\Delta + 2A\cdot D + \lambda : H^1_0(\Omega) \to H^{-1}(\Omega).\]
Apply Lemma \ref{second order equation} to get a scalar solution $u \in H^1(\Omega)$ to
\[(-\Delta + 2A\cdot D + \lambda) u = 0\]
of the form
\[u = u_0 + \tilde u_1,\]
with $\tilde u_1 \in H^1_0(\Omega)$ and $\|\tilde u_1\|_{H^1}\leq CN^{-1/2}$.
We seek solutions to $\mathcal{L}_V U = 0$ of the form $U = (P_0(D+A)-Q_I) u + \tilde R$. Plugging this ansatz into the equation and using \eqref{square of operator} we get that $\tilde R$ satisfies
\[\mathcal{L}_V \tilde R = (\lambda - \tilde{W})u.\]
By Lemmas \ref{behaviour of u0} and \ref{second order equation} we see that $\|u\|_{L^2}\leq CN^{-1/2}$.

If one had well-posedness for the boundary value problem for $\mathcal{L}_V$, we could solve for $\tilde R$ uniquely and obtain the estimate $\|\tilde R\|_{L^2}\leq C N^{-1/2}$. More generally, Proposition \ref{prop:inhomogeneous_eq_l2}, with the choices $\varphi(x) = x_1$ and $h = h_0$ where $h_0 = h_0(A,q)$ is the upper bound for $h$, gives a solution $\tilde{R} = e^{-x_1/h_0} \hat{R}$ satisfying $\norm{\tilde{R}}_{L^2} \leq C \norm{\hat{R}}_{L^2} \leq C \norm{e^{x_1/h_0} (\lambda - \tilde{W}) u}_{L^2} \leq C N^{-1/2}$. Writing
\[U = P_0(D) u_0 + (P_0(A) - Q_I)u_0 + (P_0(D+A)-Q_I) \tilde u_1 + \tilde R\]
and using the estimates of Lemmas \ref{behaviour of u0} and \ref{second order equation}, we have the desired form for the solution $U$.
\end{proof}

\begin{proof}
(of Proposition \ref{boundary determination}) We may assume, after gauge transformations if necessary, that the normal components of $A_j$ at the boundary are null for $j = 1,2$. Let $U_j$ ($j = 1,2$) be solutions to $\mathcal{L}_{V_j} U_j = 0$ constructed in Proposition \ref{solutions concentration on boundary}. By the assumption that $C_{V_1} = C_{V_2}$, we have as in Lemma \ref{lemma:integral_identity} the orthogonality condition 
\begin{align*}
0 &= \int_\Omega U_2^*(V_1 - V_2) U_1 \,dx \\
 &= N^2\int_\Omega P_0(\hat t - i\nabla\rho)(V_1 - V_2)P_0(\hat t + i\nabla\rho)\eta_M^2 e^{-2N\rho} \,dx + O(M^{-1}).
\end{align*}
Multiplying by $M^2 N^{-1}$, taking $M\to\infty$, and using Lemma \ref{behaviour of u0} we get that at the origin
\begin{eqnarray}
\label{at zero}
0 = P_0(\hat t - i\nabla\rho)P_0(A)P_0(\hat t + i\nabla\rho) + P_0(\hat t - i\nabla\rho) Q P_0(\hat t + i\nabla\rho)
\end{eqnarray}
where $A := (A_1 - A_2)(0)$ and $Q := (Q_1 - Q_2)(0)$. Writing $\zeta = \hat{t} + i\nabla \rho(0)$ and applying the matrix identities \eqref{matrix identities} in \eqref{at zero}, we get that 
\begin{eqnarray}
\label{at zero 2}
0 = -P_0(A) P_0(\bar{\zeta})P_0(\zeta) + 2(A\cdot\hat t) P_0(\zeta) + Q_I P_0(\bar{\zeta})P_0(\zeta).
\end{eqnarray}
This is true also when $\hat t$ is replaced by $-\hat t$, so we have
\begin{eqnarray}
\label{at zero 3}
0 = - P_0(A) P_0(\zeta)P_0(\bar{\zeta}) + 2(A\cdot\hat t) P_0(\bar{\zeta}) + Q_I P_0(\zeta)P_0(\bar{\zeta}).
\end{eqnarray}
Adding (\ref{at zero 2}) and (\ref{at zero 3}) together and using $\abs{\zeta}^2 = 2$ we see that 
\begin{eqnarray}
\label{at zero 4}
0 = - P_0(A) + (A \cdot \hat{t}) P_0(\hat{t}) + Q_I.
\end{eqnarray}
If $A$ were nonzero one could take $\hat{t} = A/\abs{A}$ and obtain $Q_I = 0$, and then choosing $\hat{t} \in T_0(\partial \Omega)$ orthogonal to $A$ would give $A = 0$. Therefore, $A = 0$, and going back to \eqref{at zero 4} gives $Q = 0$.
\end{proof}

%\end{comment}

%\nocite{*}
%\addcontentsline{toc}{chapter}{Bibliography}
%\bibliography{rough_dirac}
%\bibliographystyle{hamsplain}

\providecommand{\bysame}{\leavevmode\hbox to3em{\hrulefill}\thinspace}
\providecommand{\href}[2]{#2}

\end{document}